\documentclass[a4paper,12pt]{amsart}

\usepackage{amsthm,amssymb,amsmath,amscd}
\usepackage{amsthm}
\usepackage{pgf, tikz}
\usepackage{graphicx,subcaption}
\usepackage{verbatim}
\theoremstyle{definition}
\theoremstyle{plain}

\newtheorem*{lemman}%[definition]
{Lemma}
\newtheorem*{theoremn}%[definition]
{Theorem}

{Conjecture}

\theoremstyle{remark}

\newcommand{\pp}{P}

\begin{document}

\title[Ramsey number for 3-paths]
{On multicolor Ramsey number for 3-paths of length three}

\author{Tomasz \L{u}czak}

\address{Adam Mickiewicz University,
Faculty of Mathematics and Computer Science
ul.~Umultowska 87,
61-614 Pozna\'n, Poland}

\email{\tt tomasz@amu.edu.pl}

\author{Joanna Polcyn}

\address{Adam Mickiewicz University,
Faculty of Mathematics and Computer Science
ul.~Umultowska 87,
61-614 Pozna\'n, Poland}

\email{\tt joaska@amu.edu.pl}

\thanks{The first author partially 
supported by NCN grant 2012/06/A/ST1/00261. }



\date{November 24, 2016}

\begin{abstract}
We show that if we color the hyperedges of the complete $3$-uniform complete graph on $2n+\sqrt{18n+1}+2$ vertices 
with $n$ colors, then one of the color classes contains a loose path of length three.
\end{abstract}

\maketitle

%\section{Introduction}
Let $\pp$ denote the $3$-uniform path of length three by which we mean the only 3-uniform connected hypergraph 
on seven vertices with the degree sequence $(2,2,1,1,1,1,1)$. By $R(\pp;n)$ we denote the 
multicolored Ramsey number for $\pp$ defined as the smallest number $N$ such that each 
coloring of the hyperedges of the complete 3-uniform hypergraph $K^{(3)}_N$ with $n$ colors
leads to a monochromatic copy of $\pp$. It is easy to check 
 that $R(\pp;n)\ge n+6$ (see, for instance, \cite{GR, J}),
and it is believed that in fact the equality holds, i.e. 
\begin{equation*}\label{eq1}
R(\pp;n)= n+6.
 \end{equation*}
  Gy\'{a}rf\'{a}s and Raeisi \cite{GR} proved, among many other results, that $R(\pp;2)=8$. Their theorem was extended by Omidi and Shahsiah~\cite{Omidi} to loose paths of arbitrary lengths, but still only for the case of two colors. 
On the other hand, in a series of papers \cite{J, JPR, PR, P} it was verified that
 $R(\pp;n)=n+6$ for all $3\le n\le 10$.
  
%  So far the above conjecture has been confirmed for all $n\le 10$ (see \cite{GR,J,JPR,P,PR}).  
Note that from the fact that for $N\ge 8$ the largest 
$\pp$-free $3$-uniform hypergraph on $N$ vertices contains at most $\binom {N-1}2$ edges
(see \cite{JPRt}), 
it follows that for $n\ge 7$ we have
\begin{equation*}\label{eq2}
R(\pp;n)\le 3n +1.
 \end{equation*} 
Our main goal is to  improve the above  bound.

\begin{theoremn}\label{thm:main}
$R(\pp;n)\le 2n+\sqrt{18n+1}+2$.
\end{theoremn}

%\section{Main Lemma}

Let $C$ denote the (loose) $3$-uniform $3$-cycle, i.e. the only $3$-uniform linear hypergraph with six  vertices and three hyperedges. Furthermore, let  $F$ be 
the $3$-uniform hypergraph on vertices $v_1,v_2,v_3,v_4,v_5$ such that the first four of these 
vertices span in $F$ the clique, and $v_5$ is contained in the following three hyperedges:
$v_1v_2v_5$, $v_2v_3v_5$, and $v_3v_4v_5$. The following fact will be crucial for our argument.

 \begin{lemman}\label{l:main}
Let $H$ be a $3$-uniform $\pp$-free hypergraph on $n\ge 5$ vertices. Then we can delete 
from $H$ fewer than $3n$ hyperedges in such a way that the resulting hypergraph contains no copies of $C$ and $F$.
\end{lemman}     

\begin{proof} Let us first consider components containing $C$. Jackowska, Polcyn and Ruci\'nski \cite{JPR} showed
that each such component of $H$ on $n_i$ vertices has at most $3n_i-8< 3n_i$ hyperedges, 
provided $n_i\ge 7$. Furthermore, from the complete 3-hypergraph on $n_i=6$ vertices it is
enough to delete $10<3n_i$ edges to get a star, which clearly contains no copies of $C$ and $F$. Hence, to get rid of all copies of $C$ (and $F$ in components containing $C$) it is enough to remove at most 
$3n'$ of hyperedges from components containing them, where $n'$ denotes the number of vertices 
in these components combined. Now let us consider components containing $F$ but not $C$. It is easy to check 
by a direct inspection that any hyperedge $e$ which shares with $F$ just one vertex
would create a copy of $P$. Moreover, any hyperedge $e'$ which shares with  $F$ two vertices 
would create a copy of $C$. Consequently, each copy of $F$ in a $\pp$-free, $C$-free
3-uniform hypergraph $H'$ is a component. Note that each such component has  at most $\binom 53=10$ edges
and we can destroy $F$ by removing from it $4< 3\cdot 5$ hyperedges. Thus, one can 
delete from $H$ fewer than $3n$ hyperedges to destroy all copies of $C$ and $F$.
\end{proof}

\begin{proof}[Proof of Theorem] Consider a coloring of the edges of the complete 3-uniform hypergraph on 
$2n+m$ vertices with $n$ colors. Assume that no color class contains a copy of $P$. Then, 
by Lemma, we can mark as `blank' fewer than $r=3(2n+m)n$ hyperedges of the graph in such a way that when we ignore blank hyperedges
the coloring contains no monochromatic copies of $C$ and $F$. 

Let us color a pair of vertices $vw$  with a color $s$,  $s=1,2,\dots,n$, if there exist at least  three hyperedges of color $s$,  $s=1,2,\dots,n$, which contain this pair. If there are many such colors we choose any of them; if there are none we leave $vw$ uncolored. Note that every 
uncolored pair must be contained in at least $m-2$ blank hyperedges. Consequently, 
fewer than $3r/(m-2)$ pairs remains uncolored. But then there exists a color $t$, $t=1,2,\dots,n$, such that there are more than 
$$\bigg[\binom{2n+m}{2}-\frac{3r}{m-2}\bigg]\bigg/n=2n+m+(2n+m)\Big[\frac{m-1}{2n}-\frac{9}{m-2}\Big] $$
pairs colored with $t$.  If $m\ge \sqrt{18n+1}+2$, then  
$$\frac{m-1}{2n}-\frac{9}{m-2}> 0\,, $$ 
and the graph $G_t$ spanned by these pairs has more edges than vertices. 
But it means that $G_t$ contains a path of length 3, i.e. there are vertices  
$v_1,v_2, v_3,v_4$ and a color $t$ such that each of the three pairs $v_1v_2$, $v_2v_3$,
$v_3v_4$ is contained in at least three different hyperedges colored with $t$. We shall show 
that it leads to a contradiction. 

Indeed, let $H_t$ be a hypergraph spanned by hyperedges colored with the $t$th color. 
Observe first that since $v_2v_3$ is contained in three different hyperedges of $H_t$ there must be 
one which is different from $v_1v_2v_3$ and $v_2v_3v_4$; let us call it $v_2v_3v_5$ where 
$v_5\neq v_1,v_2,v_3,v_4$. Furthermore, $v_1v_2$ must be contained in a hyperedge $v_1v_2w$ of $H_t$ 
where  $w\neq v_3,v_5$, while $v_3v_4$ is contained in  some $v_3v_4u$, 
where $u\neq v_2,v_5$. Note now
that if $w\neq v_4$ and $u\neq w,v_1$, then $H_t$ contains a copy of $P$ which contradicts the fact that it is $P$-free. The case  $w=u\neq v_1,v_4$ would lead to a cycle $C$, as well as 
the cases $w=v_4$, $u\neq v_1$, and $u=v_1$, $w\neq v_4$. Finally, if the only  possible choices 
for $w$ and $u$ are $w=v_4$ and $u=v_1$, then the vertices $v_1,\dots, v_5$ span a copy of $F$, 
so we arrive at the contradiction again.     
\end{proof}

{\sl Remark}
The bound $3n$	given in Lemma is rather crude. Using the fact that each component of $H$ on $n_i$ vertices containing $C$ and two disjoint edges has at most $n_i+5$ hyperedges, provided $n_i\ge 7$ (see \cite{P}) and by careful analysis of 3-uniform intersecting families (see \cite{EKR,HK,HM,KM,PRIF}) one can improve it by a constant factor and, 
consequently, improve by a constant factor the second order term in the estimate for 
$R(P;n)$. 

{\bf Acknowledgement.} We wish to thank Andrzej Ruci\'nski for stimulating conversations.

\bibliographystyle{plain}

\end{document}